\documentclass[a4paper,12pt]{article}
\usepackage{amsmath}
\usepackage{amsthm}
\usepackage{epsfig}
\usepackage{psfrag}
\usepackage[mathscr,mathcal]{eucal}
\usepackage{amssymb}
\usepackage{enumerate}
\def \N {\mathbb N}

\def \R {\mathbb R}

\def \Ordo {{\cal O}}
\def\ordo{o}

\def \ind{1\!\!1}

\def\ul{\underline}

\def\phi{\varphi}

\def\be{\begin{equation}}
\def\ee{\end{equation}}

\def\bea{\begin{eqnarray}}
\def\eea{\end{eqnarray}}

\newcommand{\abs}[1]{\left|{#1}\right|}

\newcommand{\prob}[1]{\ensuremath{\mathbf{P}\left(#1\right)}}
\def \probp{\mathbf{P}}
\newcommand{\expect}[1]{\ensuremath{\mathbf{E}\left(#1\right)}}
\newcommand{\var}[1]{\ensuremath{\mathbf{D}^2\left(#1\right)}}
\newcommand{\cov}[2]{\ensuremath{\mathbf{Cov}\left(#1,#2\right)}}
\newcommand{\condprob}[2]{\ensuremath{\mathbf{P}\left(#1\,\big|\,#2\right)}}
\newcommand{\condexpect}[2]{\ensuremath{\mathbf{E}\left(#1\,\big|\,#2\right)}}

\def\cG{\mathcal{G}}

\def \mypoi{\mathbf{p}}

\def \mygamma{\mathbf{g}}

\def \toind {\buildrel {\text{d}}\over{\longrightarrow}}
\def \toinp {\buildrel {\text{p}}\over{\longrightarrow}}

\newcommand{\semmi}[1]{}

\def \M       {\mathcal{M}}
\def \A       {\mathcal{A}}

\newtheorem {theorem}{Theorem}
\newtheorem {lemma}{Lemma}[section]

\newtheorem {definition}{Definition}[section]

\title{Multigraph limit of the dense configuration model \\ and the preferential attachment graph}

\author{ Bal\'azs R\'ath \thanks{ETH Z\"urich, Department of Mathematics, R\"amistrasse 101, 8092 Z\"urich. \newline
Email: rathb@math.ethz.ch.}
\and  L\'aszl\'o Szak\'acs \thanks{ E\"otv\"os Lor\'and University, Institute of Mathematics,  P\'azm\'any P\'eter s\'et\'any 1/C,  
  1117 Budapest. Email:  szakacsl@cs.elte.hu } }

\begin{document}

\maketitle

\bigskip

\begin{abstract}
The configuration model is the most natural model to generate a random multigraph with a given degree sequence.
 We use the notion of dense graph limits to characterize the special form of limit objects of convergent sequences of configuration
models. We apply these results to calculate the limit object corresponding to the dense preferential attachment graph and 
the edge reconnecting model. Our main tools in doing so are (1) the relation between the theory of graph limits and that of partially exchangeable
random arrays (2) an explicit construction of our random graphs that uses urn models.
\end{abstract}

$ $

\section{Introduction}
\label{introduction}

The notion of dense graph limits was introduced in \cite{Lovasz_Szegedy_2006} and has been further developed over the years, see
\cite{very_large_graphs} for a recent survey. Heuristically, the theory of dense graph limits gives a compact way to 
characterize the statistics of a randomly chosen
small subgraph of a large dense graph.
 In \cite{randomlygrown}
the graph limits of various sequences of random dense graphs were calculated and in this paper we proceed with the investigation of this
 topic.

Our objects of study are multigraphs rather than simple graphs, i.e.\ we allow parallel and loop edges: this choice makes the definition of
 the limit objects of convergent multigraph sequences (\emph{multigraphons})  slightly more complicated than the limit objects of simple
graph sequences (graphons), but on the other hand
  the multigraph models defined below  are easier to study than the corresponding simple graph models.

The simplest way to generate a random multigraph with a prescribed degree sequence  is called the \emph{configuration model}:
 we draw $d(v)$ stubs (half-edges) at each vertex $v$ and then we uniformly choose one from the the set of possible matchings of these stubs. 
 In this paper we call such random multigraphs \emph{edge stationary} (for reasons that will become clear later) and in 
Theorem \ref{thm_graph_lim_of_edge_stationary} we  characterize
the special form of limiting multigraphons that arise as the limit of random dense edge stationary multigraph sequences.
Rougly speaking, our theorem states that the number of edges connecting the vertices $v$ and $w$ has Poisson distribution
 with parameter proportional to $d(v)d(w)$.

We also investigate two random graph models which have different definitions but turn out to have the same distribution: 
\begin{itemize} 

\item The \emph{edge reconnecting model} is a random
 multigraph evolving in time.
 Denote the multigraph at time $T$ by $\cG_n(T)$, where $T=0,1,2,\dots$ and $n=\abs{V(\cG_n(T))}$ is the number of vertices.
 We denote by $m=\abs{E(\cG_n(T))}$ the number of edges (the number of vertices and edges does not change over time).
Given the multigraph $\cG_n(T)$ we get $\cG_n(T+1)$ by uniformly choosing an edge in $E(\cG_n(T))$, 
choosing one of the endpoints of that edge with a coin flip and reconnecting the edge to a new endpoint 
which is chosen using the rule of linear preferential attachment: a vertex $v$ is chosen with probability
 $\frac{d(v)+\kappa}{2m+n\kappa}$, where $d(v)$ is the degree of vertex $v$ in $\cG_n(T)$ and
 $\kappa \in (0,+\infty)$  is a fixed parameter of the edge reconnecting model. We look at the unique stationary distribution of this
multigraph-valued  Markov chain which is a random multigraph on $n$ vertices and $m$ edges.

\item  In Section 3.4 of \cite{randomlygrown} a random multigraph called \emph{preferential attachment graph with $n$ nodes
and $m$ edges} (briefly $\text{PAG}(n,m)$) is defined. We slightly generalize the definition to obtain  $\text{PAG}_{\kappa}(n,m)$ where 
$\kappa \in (0,+\infty)$  is a fixed parameter:
let $V=\{v_1, \dots, v_n\}$ be a set of vertices. We create a sequence $v^*_1, \dots, v^*_{2m}$  with 
elements from $V$ by starting with the empty
sequence and appending random elements of $V$ one by one. If the current length of the sequence is $L$ then we choose the next 
element $v^*_{L+1}$ to be equal to 
 $v \in V$ with probability  $\frac{ d(v)+\kappa}{ L+ n\kappa}$, where $d(v)$ is the multiplicity of $v$ in the sequence $v^*_1, \dots, v^*_{L}$.
Now we create the random multigraph $\text{PAG}_{\kappa}(n,m)$ on the vertex set $V$ by adding the edges of form $\{v^*_{2k-1}, v^*_{2k}\}$ 
for each $k=1,\dots,m$.
\end{itemize}

Lemma \ref{lemma_stationary_distribution} states that the above described two random multigraphs have the same distribution.
In Theorem \ref{thm_stac_graphlim} we give the limiting multigraphon of this random multigraph when $n \to \infty$ 
 and $m \approx \frac12 \rho n^2$, where $\rho \in (0,+\infty)$ 
 is a fixed parameter of the model called the \emph{edge density}. Roughly speaking, the limiting multigraphon can be described as follows:
 it is edge stationary, and the rescaled degrees of vertices have Gamma distribution
with parameters depending on $\kappa$ and $\rho$.

$ $

The precise statements of these theorems along with the necessary notations can be found in Section \ref{basic}. 
We end the Introduction with mentioning a
few related results:

The configuration model is a random multigraph, but if we condition it to have no multiple and loop edges, then the resulting random simple
graph is uniformly distributed given its degree sequence.
In \cite{diaconis_chatterjee_sly}
the description of the limiting graphon of such sequences of simple dense graphs (and a
continuous version of the Erd\H{o}s-Gallai characterization of degree sequences) is given.

In \cite{RB_evolution} we give a characterization of the time evolution of the edge reconnecting model, viewed through the prism of the theory of
multigraphons: roughly speaking, if we start the edge reconnecting model from an arbitrary initial multigraph, then we have to run our process for
  $ n^2 \ll T $ steps until $\cG_n(T)$ becomes ``edge stationary'' and run it for 
$ n^3 \ll T$ steps until $\cG_n(T)$  becomes ``stationary''.

$ $

\noindent
{\bf Acknowledgement.}
The authors thank L\'aszl\'o Lov\'asz for posing the research problem that became the subject of this paper.

The research of Bal\'azs R\'ath was  partially supported by the OTKA (Hungarian
National Research Fund) grants
K 60708 and CNK 77778, Morgan Stanley Analytics Budapest and  Collegium Budapest and
 the grant ERC-2009-AdG 245728-RWPERCRI.
The research of L\'aszl\'o Szak\'acs was  partially supported by the OTKA grant NK 67867.

\section{Notation and results}\label{basic}


Denote by  $\N_0=\{0,1,2,\dots\}$,  $[n]:=\{1,\dots,n \}$ and $[k..n]:=\{k,\dots,n\}$. If $H_1$ and $H_2$ are arbitrary sets, 
denote by $f: \, H_1 \hookrightarrow H_2$ a generic injective function from $H_1$ to $H_2$.
Denote by $\M$ the set of undirected multigraphs (graphs with multiple and loop edges) and by $\M_n$ the set of multigraphs on $n$ vertices.
Let $G \in \M_n$. The adjacency matrix of a labeling of the multigraph $G$ with $[n]$
is denoted by $\left(B(i,j)\right)_{i,j=1}^n$, where $B(i,j) \in \N_0$ is the
number of edges connecting the vertices labeled by $i$ and $j$.
$B(i,j)=B(j,i)$ since the graph is undirected
 and  $B(i,i)$ is  two times the number of loop edges at vertex $i$ (thus $B(i,i)$ is an even number).

 We denote the set of adjacency matrices of multigraphs on $n$ nodes by $\A_n$, thus
 \[\A_n = \left\{  B\in \N_0^{n \times n}\, :\, B^T=B, \, \forall \, i \in [n] \,\;\;
   2\, |\,   B(i,i)  \right\}. \]

The degree of the vertex labeled by $i$ in $G$ with adjacency matrix $B \in \A_n$ is defined by
$ d(B,i):= \sum_{j=1}^n B(i,j)$,
thus $d(B,i)$ is the number of stubs at $i$ (loop edges count twice). Let
\begin{equation*}
m=m(G)=m(B)= \frac{1}{2} \sum_{i, j=1}^n B(i,j)=\frac{1}{2} \sum_{i=1}^n d(B,i)
\end{equation*}
 denote the number of edges. 
Denote by $\A_n^m$ the set of adjacency matrices on $n$ vertices with $m$ edges.


 An unlabeled multigraph is the equivalence class of labeled multigraphs where two 
labeled graphs are equivalent if one can be
 obtained by relabeling the other.
  Thus $\M$ is the set of these equivalence classes of labeled multigraphs, which are also called isomorphism types.

  Suppose $F \in \M_k,$  $G \in \M_n$ and denote by $A \in \A_k$ and $B \in \A_n$ the adjacency matrices of $F$ and $G$.
  If  $g: \M \to \R$ then we say that $g$ is a multigraph parameter. Let
$g(A):=g(F)$. 
Conversely, if $g: \bigcup_{k=1}^{\infty} \A_k \to \R$ is constant on isomorphism classes,
 then $g$ defines a multigraph parameter.

\subsection{Multigraphons and multigraph convergence}

We  define the \emph{induced homomorphism density} of $F$ into $G$ by
  \begin{equation}\label{hom_ind}
    t_{=}(F,G):=t_{=}(A,B):= \frac{1}{n^k} \sum_{\varphi:[k] \rightarrow [n]}
     \ind \left[ \, \forall i,j \in [k]:\; A(i,j)= B(\varphi(i),\varphi(j))\right] .
  \end{equation}

The notion of convergence of simple graph sequences and
 several equivalent characterizations of \emph{graphons} 
(limit objects of convergent graph sequences) were given in \cite{Lovasz_Szegedy_2006}.
In \cite{KI_RB} a natural generalization of the theory of dense graph limits to multigraphs is given 
(see also \cite{Lovasz_Szegedy_2010} for similar results in a more general setting).
We say that a sequence of multigraphs  $\left(G_n\right)_{n=1}^{\infty}$ is convergent
if for every $k \in \N$ and every multigraph $F \in \M_k$ the limit $g(F)=\lim_{n \to \infty}t_{=}(F,G_n)$ exists, 
moreover 
 we have $\sum_{A \in \A_k} g(A)=1$. 
 The limit object of a convergent multigraph sequence is a measurable function
 $W:[0,1] \times [0,1] \times \N_0 \to [0,1]$  satisfying

\begin{equation} \label{multigraphon_def_eq_intro}
W(x,y,k) \equiv W(y,x,k), \quad
 \sum_{k=0}^{\infty} W(x,y,k) \equiv 1, \quad
 W(x,x, 2k+1) \equiv 0.
\end{equation}
Such functions are called \emph{multigraphons}. 
 For every multigraphon $W$ and multigraph $F \in \M_k$ with adjacency matrix $A \in \A_k$ we define
\begin{equation}
\label{def_grafon_ind_hom_sur}
 t_{=}(F,W):=  t_{=}(A,W):=  \int_{[0,1]^k} \prod_{i\leq j \leq k}
W(x_i,x_j,A(i,j))\, \mathrm{d} x_1\, \mathrm{d} x_2\, \dots\, \mathrm{d}x_k
\end{equation}

We say that $G_n \to W$ if for every $k \in \N$ and every $F \in \M_k$ we have
 \[\lim_{n \to \infty}t_{=}(F,G_n)=t_{=}(F,W).\] 
 Theorem 1 of \cite{KI_RB} states that if a sequence of multigraphs $\left(G_n\right)_{n=1}^{\infty}$ is 
convergent then $G_n \to W$ for
some multigraphon $W$ and conversely, every multigraphon $W$ arises this way.
 The limiting multigraphon of a convergent sequence is not unique,
 but if we define the equivalence relation 
$W_1 \cong W_2$ by $\forall\, F \in \M :\, t_{=}(F,W_1)=t_{=}(F,W_2)$ then obviously
 $G_n \to W_1$,  $G_n \to W_2 $ implies $W_1 \cong W_2$. For other characterisations
of the equivalence relation $\cong$ for graphons, see \cite{lovasz_uniqueness}.

For a multigraphon $W$ and $x \in [0,1]$ we define the \emph{average degree} of
 $W$ at $x$ and the \emph{edge density} of $W$ by
\begin{align}
\label{def_degree_multtigraphon}
D(W,x) &:= \int_0^1 \sum_{k=0}^{\infty} k\cdot W(x,y,k)\, \mathrm{d}y,  \\
\label{def_edge_density_multigraphon}
\rho(W) &:= \int_0^1 \int_0^1 \sum_{k=0}^{\infty} k\cdot W(x,y,k)\, \mathrm{d}y\, \mathrm{d}x.
\end{align}
If $\rho(W)<+\infty$ then $D(W,x)<+\infty$ for Lebesgue-almost all $x$.

Given a multigraphon $W$ we define the \emph{degree distribution function} of $W$ by 
\begin{equation}\label{def_eq_F_W}
F_W(z) = \int_0^1 \ind [ D(W,x) \leq z ] \, \mathrm{d}x, \qquad z \geq 0.
\end{equation}
Indeed, $F_W(\cdot)$ is a probability distribution function 
on $[0,\infty)$, i.e.\ it is nonnegative, right continuous, increasing
and satisfies $\lim_{z \to \infty} F_W(z) =1$. It is easy to see that we have
$\rho(W)= \int_0^{\infty} z \, \mathrm{d}F_W(z)$.  Denote  by
\begin{equation}\label{def_eq_F_W_inverse}
F_W^{-1}(u):= \min \{z: F_W(z)  \geq u \}, \qquad u \in (0,1). 
\end{equation}

\subsection{Random multigraphs and random adjacency matrices}

We denote a random element of $\A_n$ by $\mathbf{X}_n$. 
We may associate a random multigraph $\cG_n$  to $\mathbf{X}_n$
by taking the isomorphism class of $\mathbf{X}_n$.

We say that a sequence of random multigraphs  $\big(\cG_n\big)_{n=1}^{\infty}$ converges in probability  to
a  multigraphon $W$ (or briefly write $\cG_n \toinp W$) if
 for every multigraph $F$ we have $t_{=}(F,\cG_n) \toinp t_{=}(F,W)$, i.e.
\begin{equation}\label{convergence_of_random_multigraphons}
 \forall \, F \in \M \; \forall \, \varepsilon>0:  \;
\lim_{n \to \infty} \prob{ \abs{ t_{=}(F, \cG_n) - t_{=}(F, W) }>\varepsilon}=0.
\end{equation}
We say that $\mathbf{X}_n \toinp W$ if $\cG_n \toinp W$ holds for the associated random multigraphs.

Note that the definitions of the edge reconnecting model and the $\text{PAG}_{\kappa}$ (see Section \ref{introduction})
in fact naturally give rise to a random labeled graph, i.e.\ a random element $\mathbf{X}_n$ of $\A_n$.
The edge reconnecting Markov chain is easily seen to be  irreducible and aperiodic 
on the state space $\A_n^m$, thus the stationary
distribution is indeed unique.

We say that the distribution of $\mathbf{X}_n$ is \emph{edge stationary} if the conditional distribution
of  $\mathbf{X}_n$ given the degree sequence $\left(d(\mathbf{X}_n,i)\right)_{i=1}^n$ is the same as that of
the configuration model (see Section \ref{introduction}) with the same degree sequence.

 Recall the formulas defining the Poisson and Gamma distributions:
\begin{align}
\label{def_mypoi}
\mypoi(k,\lambda)&:=e^{-\lambda}\frac{\lambda^k}{k!}\\
\label{def_mygamma}
\mygamma(x,\alpha,\beta)&:=x^{\alpha-1} \frac{\beta^{\alpha}
e^{-\beta x}}{\Gamma(\alpha)} \ind[x>0]
\end{align}
We say that a nonnegative integer-valued random variable $X$ has Poisson distribution with parameter
 $\lambda$ (or briefly denote $X \sim \text{POI} \left( \lambda \right)$) if $\prob{X=k}=\mypoi(k,\lambda)$ for
all $k \in \N$. 
 We say that a nonnegative real-valued random variable $Z$ 
has gamma distribution with parameters $\alpha$ and 
$\beta$ (or briefly denote $Z \sim \text{Gamma}(\alpha, \beta)$) if $\prob{Z \leq z}=\int_0^z \mygamma(x,\alpha,\beta) \, \mathrm{d} x  $.

For a real-valued nonnegative random variable $X$ define
\[\expect{ X ; m}:= \expect{ X \cdot \ind[ X \geq m]}.\]

\subsection{Statements of main results}

First we state our theorem characterizing the form of multigraph limits of edge stationary multigraph sequences:

\begin{theorem}\label{thm_graph_lim_of_edge_stationary}
Let $W$ denote a multigraphon with $\rho(W)<+\infty$. 
If $\mathbf{X}_n$ is an $\A_n$-valued edge stationary random variable for all $n \in \N$, $\mathbf{X}_n \toinp W$ 
for some multigraphon $W$, and the sequence $\left( \mathbf{X}_n \right)_{n=1}^{\infty}$ satisfies 
\begin{align}
\label{egyenletes_int_zuri_1}
\lim_{m \to \infty} \sup_{n \in \N} \frac{1}{\binom{n}{2}} \sum_{i<j \leq n} \expect{ X_n(i,j) ; m} &= 0 \\
\label{egyenletes_int_zuri_2}
\lim_{m \to \infty} \sup_{n \in \N} \frac{1}{n} \sum_{i=1}^n \expect{ X_n(i,i) ; m} &= 0,
 \end{align}
then the limiting multigraphon $W$ can be rewritten in the form $\hat{W} \cong W$ where
\begin{equation}\label{edge_stac_lim_graphon}
  \hat{W}(x,y,k)\stackrel{\eqref{def_eq_F_W_inverse}, \eqref{def_mypoi} }{:=}
\left\{
\begin{array}{ll}
\mypoi(k,\frac{F_W^{-1}(x)F_W^{-1}(y)}{\rho(W)}) & \mbox{ if $x \neq y$}\\
\ind [2\,|\,k]\cdot \mypoi \left(\frac{k}{2},\frac{F_W^{-1}(x)F_W^{-1}(y)}{2\rho(W)} \right) & \mbox{ if $x= y$}
\end{array} \right.
\end{equation}

\end{theorem}

Now we state our results describing the multigraph limit of the $\text{PAG}_{\kappa}(n,m)$ and
 the stationary distribution of the edge reconnecting model.

For an adjacency matrix $B \in \A_n$ denote by
$m'(B)=\sum_{i=1}^n \sum_{j=1}^{i-1} B(i,j)$
the number of non-loop edges of the corresponding graph.

\begin{lemma}\label{lemma_stationary_distribution}
 The unique stationary  distribution of the edge reconnecting model with linear preferential attachment parameter $\kappa$ and
 state space $\A_n^m$
 has the same distribution as \emph{$\text{PAG}_{\kappa}(n,m)$}. If  $\mathbf{X}_n$ has this distribution then for all $B \in  \A_n^m$ 

 \begin{equation}  \label{stationary}
 \prob{ \mathbf{X}_n=B}  =  \frac{\prod_{i=1}^n
\prod_{j=1}^{d(B,i)}(\kappa+j-1)}{ \prod_{j=1}^{2m}(\kappa n
+j-1)}
 \frac{m! 2^{m'(B)}}{\left(\prod_{i=1}^n \prod_{j=1}^{i-1} B(i,j)!\right)
  \left(\prod_{i=1}^n \frac{B(i,i)}{2}! \right)}
 \end{equation}
\end{lemma}

At the end of Section 3.4 of \cite{randomlygrown} the following theorem is stated:

Let $\text{SPAG}(n,m)$ denote the simple graph obtained from $\text{PAG}(n,m)$ by deleting loops and keeping only one copy of the parallel edges. Then
\begin{equation}\label{spag_thm}
\text{SPAG}\left(n, \frac{n^2}{2} \cdot (\rho +\ordo(1))\right) \toinp W_s,
\qquad W_s(x,y):=1-\exp(-\rho \ln(x)\ln(y)),
\end{equation}
where (analogously to \eqref{convergence_of_random_multigraphons}) the symbol 
$\toinp$ denotes convergence in probability of 
a sequence of random simple graphs to a (simple) graphon.

It is easy to see that \eqref{spag_thm} is a corollary of the following theorem:

\begin{theorem}\label{thm_stac_graphlim}
Let us fix $\kappa, \rho \in (0,+\infty)$.
If $\mathbf{X}_n$ is a random element of $\A_n^{m(n)}$ with distribution \eqref{stationary} for $n=1,2,\dots$, moreover the asymptotic edge
density is
\[\lim_{n \to \infty} \frac{2m(n)}{ n^2}=\rho,\] then $\mathbf{X}_n \toinp W$ where

 \begin{equation}  \label{stac_poi}
W(x,y,k) =
 \left\{
\begin{array}{ll}
\mypoi(k,\frac{F^{-1}(x)F^{-1}(y)}{\rho}) & \mbox{ if $x \neq y$}\\
\ind [2|k]\cdot \mypoi \left(\frac{k}{2},\frac{F^{-1}(x)F^{-1}(y)}{2\rho} \right) & \mbox{ if $x= y$}
\end{array} \right.
 \end{equation}
and $F^{-1}$ is the inverse function of 
$F(z)=\int_0^z \mygamma(y,\kappa,\frac{\kappa}{\rho}) \mathrm{d}y$, see \eqref{def_mygamma}.
\end{theorem}

 Note the similarity of the multigraphons appearing in \eqref{edge_stac_lim_graphon}
and \eqref{stac_poi}: as we will see later,
 this is a consequence of the fact that the distribution of
 $\text{PAG}_{\kappa}(n,m)$ is edge stationary.

The proofs of the above stated theorems rely on the following ideas:
\begin{itemize}
 \item We relate our random graph models to urn models with multiple colors (e.g. the well-known P\'olya urn model):
the number of balls is $2m$ and  they are colored with $n$ possible colors. Each ball corresponds to a stub, each color
corresponds to a labeled vertex and the edge set of the multigraph depends on the positions of balls in the urn.
\item We make use of the underlying symmetries of the distributions of our random graphs by relating the
 theory of graph limits to the theory of partially exchangeable arrays of random variables,
 a connection first observed in \cite{diaconis_janson}.
\end{itemize}

The rest of this paper is organized as follows: 

In Section \ref{section_vertex_exch} we introduce
the notion of random, vertex exchangeable, infinite adjacency matrices as well 
as $W$-random multigraphons and deduce some useful 
results relating the convergence of these objects to graph limits.

 In Section  \ref{section_stationary_state} we relate
the notion of edge stationarity to the ball exchangeability of the corresponding urn models and prove the convergence results stated above.

\section{Vertex exchangeable arrays}\label{section_vertex_exch}

In this section we introduce random infinite arrays $\mathbf{X}=\big( X(i,j) \big)_{i,j=1}^{\infty}$ that 
arise as the adjacency matrices of random infinite labeled multigraphs and
we give probabilistic meaning to the homomorphism densities $t_=(F,W)$ 
 by introducing $W$-random infinite multigraphs $\mathbf{X}_W$.
 We also introduce the notion of  the average degree $D(\mathbf{X},i)$ of a vertex $i$ in an
 infinite, dense, vertex exchangeable multigraph.

In Subsection \ref{subsection_conv_of_arrays} give a useful alternative characterisation of $\cG_n \toinp W$ using
exchangeable arrays and prove that under certain technical conditions the average degrees of $\cG_n$ converge
in distribution to the average degrees $D(\mathbf{X}_W,i)$ of the limiting $W$-random infinite array.

$ $

 Let $\A_{\N}$ denote the set of adjacency matrices $\left(A(i,j)\right)_{i,j=1}^{\infty}$of countable  multigraphs:
 \begin{equation*}
 \A_{\N} = \left\{  A\in \N_0^{\N \times \N}\, :\, \forall \, i,j \in \N\;  A(i,j)\equiv A(j,i), \; \;\, \forall \, i \in \N\,\;
   2\, |\,   A(i,i)  \right\}.
 \end{equation*}
  We consider the probability space $\left(\A_{\N}, \mathcal{F}, \probp \right)$ where $\mathcal{F}$ is the coarsest sigma-algebra with
 respect to which $A(i,j)$ is measurable for all $i,j$ and $\probp$ is a probability measure on the measurable space
 $\left( \A_{\N}, \mathcal{F} \right)$. We are going to denote the infinite random array with distribution
 $\probp$ by $\mathbf{X} = \left(X(i,j) \right)_{i,j=1}^{\infty}$. We use the standard notation $\mathbf{X}\sim \mathbf{Y}$ if
 $\mathbf{X}$ and $\mathbf{Y}$ are 
identically distributed (i.e.\ their distribution $\probp$ is identical on $\left( \A_{\N}, \mathcal{F}\right)$).

If $\mathbf{X}$ is a random element of $\A_{\N}$, let $\mathbf{X}^{[k]}$ be the random element of $\A_k$ defined by
 $\mathbf{X}^{[k]}:=\left(X(i,j)\right)_{i,j=1}^k$.

 \begin{definition}[$W$-random infinite multigraphons]
\label{def_X_W}
 Let $\left(U_i\right)_{i=1}^{\infty}$  be independent random variables uniformly distributed in $[0,1]$.
 Given a multigraphon $W$
  we define the random countable adjacency matrix
 $\mathbf{X}_W = \left(X_W(i,j)\right)_{i,j=1}^{\infty}$ as follows: Given the background variables 
 $(U_i)_{i=1}^{\infty}$ the random variables
 $\left(X_W(i,j)\right)_{i \leq j \in \N}$ are conditionally independent and
 \[\condprob{ X_W(i,j)=m}{(U_i)_{i=1}^{\infty} }=W(U_i,U_j,m),\]
  that is
  if $A \in \A_k$ then we have
 \begin{equation}\label{X_W_indep_prod_formula}
 \condprob{ \mathbf{X}_W^{[k]}=A}{(U_i)_{i=1}^{\infty} }:=
 \prod_{i\leq j \leq k} W(U_i,U_j,A(i,j)).
 \end{equation}
 \end{definition}
In plain words: if $i \neq j$ and $U_i=x$, $U_j=y$
 then the number of multiple edges between the vertices labeled by $i$ and $j$
 in $\mathbf{X}_W$ has  distribution $\big( W(x,y,k) \big)_{k=1}^{\infty}$ and
 the number of loop edges at vertex $i$ has distribution
  $\big( W(x,x,2k) \big)_{k=1}^{\infty}$ (these are indeed proper probability distributions by \eqref{multigraphon_def_eq_intro}).

For every multigraphon $W$ and  multigraph $F \in \M_k$ with adjacency matrix $A \in \A_k$ we have
\begin{equation} \label{homind_grafon_valszamosan}
  t_{=}(F,W) \stackrel{\eqref{def_grafon_ind_hom_sur},\eqref{X_W_indep_prod_formula}}{=} \prob{ \mathbf{X}_W^{[k]}=A }.
\end{equation}

Recalling \eqref{def_degree_multtigraphon} and \eqref{def_edge_density_multigraphon} we have
\begin{equation}\label{degree_W_expect}
D(W,x)  =\condexpect{X_W(1,2)}{U_1=x},  \qquad \quad
\rho(W) =\expect{X_W(1,2)}.
\end{equation}
If $\rho(W)<+\infty$ then $D(W,U_1)<+\infty$ almost surely.

We say that a random infinite array $\mathbf{X} = \left(X(i,j)\right)_{i,j=1}^{\infty}$ is \emph{vertex exchangeable} if
\begin{equation}\label{vertex_exch}
  \left(X(\tau(i),\tau(j))\right)_{i,j=1}^{\infty} \sim \left(X(i,j)\right)_{i,j=1}^{\infty}
\end{equation}
for all finitely supported permutations $\tau : \N\rightarrow \N$. We call $\mathbf{X} = \left(X(i,j)\right)_{i,j=1}^{\infty}$ 
\emph{dissociated} if for all $m,n \in \N$ the $\A_n$-valued random variable
 $\left(X(i,j)\right)_{i,j=1}^n$ is independent of the $\A_m$-valued random variable $\left( X(i,j)\right)_{i,j=n+1}^{n+m}$.

In our case an infinite exchangeable array can be thought of as the adjacency matrix of a random multigraph with vertex set $\N$: 
the adjacency matrix of this random infinite multigraph is vertex 
exchangeable if and only if the distribution of the random graph is invariant under the relabeling of the vertices and dissociated if
 and only if subgraphs spanned by disjoint vertex sets are independent.

It follows from Definition \ref{def_X_W} that $\mathbf{X}_W$ is vertex exchangeable and dissociated and by 
Aldous' representation theorem (see Theorem 1.4, Proposition 3.3 and Theorem 5.1 in \cite{aldous_exch}), the converse holds:
a random element $\mathbf{X}$ of $\A_{\N}$ is vertex exchangeable and dissociated if and only if $\mathbf{X} \sim \mathbf{X}_W$
for some multigraphon $W$. Although the notion of the $W$-random graph (see Definition \ref{def_X_W})
 is already present in \cite{Lovasz_Szegedy_2006},  the connection of Aldous' representation theorem with the 
theory of graph limits was first observed in \cite{diaconis_janson}. 
See also Theorem 3.1, Theorem 3.2, Proposition 3.4 of \cite{Lovasz_Szegedy_2009}.
For a self-contained proof of this representation theorem for multigraphons, see  Theorem 1 and Theorem 2 in \cite{KI_RB}.

For a vertex exchangeable infinite array $\mathbf{X}$ satisfying $\expect{X(1,2)}<+\infty$
 we define the average degree of $\mathbf{X}$ at vertex $i$ by
\begin{equation}\label{def_D_X_lim}
 D(\mathbf{X},i):= \lim_{n \to \infty} \frac1n \sum_{j=1}^n X(i,j).
\end{equation}
 The sum $\frac1n \sum_{j=1}^n X(i,j)$ indeed almost surely converges to a
random variable as $n \to \infty$ by de Finetti's theorem (see Section 2.1 of \cite{aldous_more}) and the conditional strong law of large numbers.
From \eqref{def_degree_multtigraphon}, Definition \ref{def_X_W} and \eqref{degree_W_expect} we get
\begin{equation}\label{degree_W_loln}
D(\mathbf{X}_W,i)= \lim_{n \to \infty} \frac1n \sum_{j=1}^n X_W(i,j) \stackrel{\text{a.s.}}{=}  D(W,U_i).
\end{equation}

\subsection{Convergence of exchangeable arrays}
\label{subsection_conv_of_arrays}
In this subsection we state and prove two lemmas: in  Lemma \ref{lemma_homkonv_konv_in_dist}  we relate convergence of dense random multigraphs to
 convergence of the probability measures of the corresponding random arrays and in Lemma 
\ref{lemma_uniform_integrabiliy_whole} we give sufficient conditions under which convergence of dense random multigraphs
imply convergence of the degree distribution of these graphs. 

$ $

We say that a sequence of random infinite  arrays $\left( \mathbf{X}_n \right)_{n=1}^{\infty}$ 
\emph{converges in distribution} 
to a random infinite  array $\mathbf{X}$ (or briefly denote  $\mathbf{X}_n \, \toind \mathbf{X}$) if
$\mathbf{X}_n^{[k]}$ converges in distribution to  $\mathbf{X}^{[k]}$ for all $k \in \N$, i.e.
\begin{equation}\label{eobankonv_detailed}
 \forall \, k \in \N, \; A \in \A_n: \; \; \lim_{n \to \infty} \prob{ A=\mathbf{X}_n^{[k]}} 
= \prob{ A=\mathbf{X}^{[k]}}
\end{equation}
 If $\mathbf{X}_n$ is vertex exchangeable for all $n$, then $\mathbf{X}$
 is also vertex exchangeable. 

Let $\mathbf{X}_n$ denote a random element of $\A_n$.
We say that the distribution $\mathbf{X}_n$ is \emph{vertex exchangeable} if
  for all permutations $\tau: [n] \to [n]$ and $B \in \A_n$
\begin{equation}\label{finite_array_exch}
\prob{ \forall\, i, j \in [n]: \; B(i,j)=X_n(i,j)}=\prob{ \forall\, i, j \in [n]:\; B(i,j)=X_n(\tau(i),\tau(j))},
\end{equation}
that is  $\left(X(i,j)\right)_{i,j=1}^n \sim \left(X(\tau(i),\tau(j))\right)_{i,j=1}^n$ holds.

 If $\mathbf{X}_n$ is a random element of $\A_n$ then
 $\mathbf{X}_n^{[k]}=\left(X_n(i,j)\right)_{i,j=1}^k$ is well-defined for $k \leq n$, thus we 
might define $\mathbf{X}_n \, \toind \mathbf{X}$ 
(where $\mathbf{X}$ is a random element of $\A_{\N}$)
by \eqref{eobankonv_detailed}. It is easy to show that if $\mathbf{X}_n$ is vertex exchangeable for each $n \in \N$ then $\mathbf{X}$ inherits
 this property.

Also note that by \eqref{homind_grafon_valszamosan} we have $\mathbf{X}_n \toind \mathbf{X}_W$
 if and only if for all $k \in \N$ and for all  $A\in \A_k$ we have 
$\lim_{n \to \infty} \prob{\mathbf{X}_n^{[k]} =A}=t_{=}(A,W)$.

\begin{lemma}\label{lemma_homkonv_konv_in_dist}
 Let $\mathbf{X}_n=\left( X_n(i,j) \right)_{i,j=1}^n$ be a random, vertex exchangeable element 
of  $\A_n$ for all $n \in \N$. The following statements are equivalent:
\begin{enumerate}[(a)]
\item $\mathbf{X}_n \toinp W$, that is
$\forall \, k \; \forall \, A \in \A_k: \;\; t_{=}(A, \mathbf{X}_n) \toinp t_{=}(A, W)$
\item  $\mathbf{X}_n \toind \mathbf{X}_W$, that is $\forall\,k \; \forall\, A\in \A_k: \; \; \lim_{n \to \infty}
\prob{\mathbf{X}_n^{[k]} =A}=t_{=}(A,W)$
\end{enumerate}
\end{lemma}
\begin{proof}
 We are going to use the fact $\lim_{n \to \infty} \frac{ n\cdot (n-1) \dots (n-k+1)}{n^k} =1$ many times in this proof.

We first prove $(a) \implies (b)$:
\begin{multline}\label{a_bol_b}
\lim_{n \to \infty} \prob{ \mathbf{X}_n^{[k]} = A} \stackrel{\eqref{finite_array_exch}}{=} 
\lim_{n \to \infty} \frac{(n-k)!}{n!} \sum_{ \varphi: [k] \hookrightarrow [n]}
\prob{ \left(X_n(\varphi(i),\varphi(j))\right)_{i,j=1}^k =A}=\\
\lim_{n \to \infty} \frac{1}{n^k} \sum_{ \varphi: [k] \to [n]}
\prob{ \left(X_n(\varphi(i),\varphi(j))\right)_{i,j=1}^k =A} 
\stackrel{\eqref{hom_ind}}{=} \lim_{n \to \infty} \expect{t_=(A,\mathbf{X}_n)}
\stackrel{(a)}{=}
t_=(A,W)
\end{multline}

Now we prove $(b) \implies (a)$: The idea of this proof comes from Lemma 2.4 of \cite{Lovasz_Szegedy_2006}.

From $(b)$ we get $\expect{t_=(A, \mathbf{X}_n)} \to t_=(A,W)$ for all $A$ by the argument used in \eqref{a_bol_b}.
 In order to have
$t_=(A, \mathbf{X}_n) \toinp t_=(A,W)$ we only need to show 
\[ \lim_{n \to \infty} \var{t_=(A, \mathbf{X}_n)}= \lim_{n \to \infty} 
\expect{t_=(A, \mathbf{X}_n)^2} - t_=(A,W)^2 =0.  \]
This follows by the computation
\begin{multline*}
 \lim_{n \to \infty} \expect{t_=(A, \mathbf{X}_n)^2} \stackrel{\eqref{hom_ind}}{=}  \\
\lim_{n \to \infty} \frac{1}{n^{2k}} \sum_{\varphi:[2k] \to [n]}
\prob{ A= \left( X_n(\varphi(i),\varphi(j)) \right)_{i,j=1}^k, \; A=\left( X_n(\varphi(i),\varphi(j)) \right)_{i,j=k+1}^{2k}}= \\
\lim_{n \to \infty} \frac{(n - 2k)! }{n!} \sum_{\varphi:[2k] \hookrightarrow [n]}
\prob{ A= \left( X_n(\varphi(i),\varphi(j)) \right)_{i,j=1}^k, \; A=\left( X_n(\varphi(i),\varphi(j)) \right)_{i,j=k+1}^{2k}}
\stackrel{\eqref{finite_array_exch}}{=}  \\
\lim_{n \to \infty} 
\prob{ A= \left( X_n(i,j) \right)_{i,j=1}^k, \; A=\left( X_n(i,j) \right)_{i,j=k+1}^{2k}}\stackrel{(b)}{=}\\
\prob{ A= \left( X_W(i,j) \right)_{i,j=1}^k, \; A=\left( X_W(i,j) \right)_{i,j=k+1}^{2k}}
\stackrel{(*)}{=} t_=(A,W)^2
\end{multline*}
In the equation $(*)$ we used  the fact that $\mathbf{X}_W$ is dissociated and \eqref{homind_grafon_valszamosan}.
\end{proof}

$ $

Recall that for a real-valued nonnegative random variable $X$ we denote
$\expect{ X ; m}:= \expect{ X \cdot \ind[ X \geq m]}$.
A sequence of  real-valued nonnegative random variables $\left(X_n\right)_{n=1}^{\infty}$ is uniformly integrable
 (see Chapter 13 of \cite{williams}) if
\begin{equation*}
\lim_{m \to \infty} \max_{n} \expect{ X_n ; m}=0.
\end{equation*}

Now we state and prove a lemma in which we give sufficient conditions under which 
$\tilde{\mathbf{X}}_n \toind \mathbf{X}$ implies $\frac{1}{n}d(\tilde{\mathbf{X}}_n,i) \toind D(\mathbf{X},i)$.
Note that some extra conditions are indeed needed, because it might happen that
 very few pairs of vertices of $\tilde{\mathbf{X}}_n$ with a huge number of
parallel edges between them remain invisible if we only sample small subgraphs of $\tilde{\mathbf{X}}_n$, but still
 cause a sigificant distortion in the distribution of the degrees of vertices in $\tilde{\mathbf{X}}_n$.
This phenomenon is related to the fact that weak convergence of a sequence of random variables $X_n \toind X$ does not
necessarily imply the convergence of the means of $X_n$ to that of  $X$: the uniform integrability of $\left(X_n\right)_{n=1}^{\infty}$
is a sufficient (and essentially necessary) condition that rules out pathological behavior.

\begin{lemma}\label{lemma_uniform_integrabiliy_whole}

$ $

\begin{enumerate}[(i)]
\item \label{lemma_uniform_integrabiliy_a}
If $\left(\mathbf{X}_n\right)_{n=1}^{\infty}$ is a sequence of infinite vertex exchangeable arrays,  the sequence \newline
$\left( X_n(1,2)\right)_{n=1}^{\infty}$  is uniformly integrable and  $\mathbf{X}_{n} \toind \mathbf{X}$,
 then  for all $k \in \N$ we have
\begin{equation}\label{matrix_and_degrees_converge}
\left( \mathbf{X}_{n}^{[k]},
 \left( D(\mathbf{X}_{n},i)\right)_{i=1}^k \right) \toind
 \left( \mathbf{X}^{[k]},
 \left( D(\mathbf{X},i)\right)_{i=1}^k \right).
\end{equation}
\item \label{lemma_uniform_integrabiliy_b}
If  $\tilde{\mathbf{X}}_n$ is a random, vertex exchangeable element of $\A_n$ for each $n \in\N$, 
$\tilde{\mathbf{X}}_n \toind \mathbf{X}$ holds for some infinite vertex exchangeable array $\mathbf{X}$
and  the sequences $\left(\tilde{X}_n(1,1)\right)_{n=1}^{\infty}$ and
 $\left(\tilde{X}_n(1,2)\right)_{n=1}^{\infty}$ are uniformly integrable then 
 for all $k \in \N$
\begin{equation}\label{matrix_and_degrees_converge_finite}
 \left( \tilde{\mathbf{X}}_n^{[k]},
 \left( \frac{1}{n} d(\tilde{\mathbf{X}}_{n},i)\right)_{i=1}^k \right) \toind
 \left( \mathbf{X}^{[k]},
 \left( D(\mathbf{X},i)\right)_{i=1}^k \right).
\end{equation}
\end{enumerate}
\end{lemma}

\begin{proof}

$ $

Proof of \eqref{lemma_uniform_integrabiliy_a}:  We first prove that \eqref{matrix_and_degrees_converge} holds if we further assume
 $\prob{X_{n}(i,j) \leq m}\equiv 1$ for some $m \in \N$. 
  By the method of moments we only need to show that for all $\mu_{i,j} \in \N_0$, $1\leq i\leq j \leq k$ and $\nu_i \in \N_0$, $1 \leq i \leq k$ 
we have
 \begin{equation}\label{degree_conv_moments1}
 \lim_{n \to \infty} \expect{
 \prod_{  i\leq j \leq k} X_n(i,j)^{\mu_{i,j}} \cdot \prod_{i=1}^k D(\mathbf{X}_{n},i)^{\nu_i}}=
 \expect{
 \prod_{  i\leq j \leq k} X(i,j)^{\mu_{i,j}} \cdot \prod_{i=1}^k D(\mathbf{X},i)^{\nu_i}}.
 \end{equation}

 For every $i \in [k]$ choose $J(i) \subseteq \N$ such that for all $i$ we have $\abs{J(i)}=\nu_i$ and $J(i) \cap [k]=\emptyset$, moreover
 for all
$i \neq i'$ we have $J(i)\cap J(i') = \emptyset$.
 In order to prove
\eqref{degree_conv_moments1} we first show that if $\prob{X(i,j) \leq m}\equiv 1$ for some $m \in \N$ then 
\begin{equation}\label{atirt_momentumok}
\expect{
 \prod_{  i\leq j \leq k} X(i,j)^{\mu_{i,j}} \cdot \prod_{i=1}^k D(\mathbf{X},i)^{\nu_i}}=
     \expect{
 \prod_{  i\leq j \leq k} X(i,j)^{\mu_{i,j}} \cdot \prod_{i=1}^k \prod_{j \in J(i)} X(i,j)}
\end{equation}
Denote by $\nu=\sum_{i=1}^k \nu_i$ and  $\ul{\nu}:= \{ (i,l)\,:\, i \in [k], \, l \in [\nu_i] \}$ and 
$\mathbf{X}^{[k],\mu}:=\prod_{  i\leq j \leq k} X(i,j)^{\mu_{i,j}}$.
 Using \eqref{def_D_X_lim} and  dominated convergence,
 the left-hand side of \eqref{atirt_momentumok} is equal to
\begin{multline*}
\lim_{ n \to \infty} 
\expect{    \mathbf{X}^{[k],\mu} \prod_{i=1}^k \left( \frac{1}{n} \sum_{j=1}^n X(i,j)  \right)^{\nu_i}} = \\
\lim_{ n \to \infty} \frac{1}{n^{\nu}}
\sum_{j: \, \ul{\nu} \to [n]} \expect{  \mathbf{X}^{[k],\mu}  \prod_{i=1}^k \prod_{l=1}^{\nu_i} X(i,j(i,l))}=\\
\lim_{ n \to \infty} \frac{1}{n^{\nu}}
\sum_{j:\, \ul{\nu} \hookrightarrow [k..n]} \expect{  \mathbf{X}^{[k],\mu}  \prod_{i=1}^k \prod_{l=1}^{\nu_i} X(i,j(i,l))}
\stackrel{\eqref{vertex_exch}}{=}\\
\lim_{ n \to \infty} \frac{1}{n^{\nu}}
\sum_{j:\, \ul{\nu} \hookrightarrow [k..n]} \expect{  \mathbf{X}^{[k],\mu}   \prod_{i=1}^k \prod_{j' \in J(i)} X(i,j') }
\end{multline*}
 Now the right-hand side of the above equation is
easily shown to be equal to the right-hand side of \eqref{atirt_momentumok}.

Having established \eqref{atirt_momentumok}, our assumptions $\mathbf{X}_n \toind \mathbf{X}$ and
$\prob{X_{n}(i,j) \leq m}\equiv 1$ imply  the equality \eqref{degree_conv_moments1}: if we rewrite both the
 left and the right hand side of \eqref{degree_conv_moments1} in the form corresponding to the right hand side of 
\eqref{atirt_momentumok}, then we only need to check that the expected value of a polynomial function of finitely many values 
of $\mathbf{X}_n$ converge, and this follows from the definition of
$\mathbf{X}_n \toind \mathbf{X}$ (for details on $\toind$, see \cite{billingsley}).

Having established \eqref{matrix_and_degrees_converge} under the condition
 $\prob{X_{n}(i,j) \leq m}\equiv 1$ we now prove \eqref{matrix_and_degrees_converge} without assuming 
this condition.
If we define the truncated array
$X^m(i,j):=\min\{ X(i,j), m\}$, then
 for each $m \in \N$ we have $\mathbf{X}^m_n \toind \mathbf{X}^m$ from which
\begin{equation}\label{truncated_toinp}
\left( \mathbf{X}_n^{m,[k]},
 \left( D(\mathbf{X}_n^m,i)\right)_{i=1}^k \right) \toind
\left( \mathbf{X}^{m,[k]},
 \left( D(\mathbf{X}^m,i)\right)_{i=1}^k \right)
\end{equation}
follows by the previous argument. By uniform integrability for every $\varepsilon>0$ there is an $m$ such that for all $n$ we have
\begin{equation}\label{unif_int_epsilon_trunc}
 \expect{ D(\mathbf{X}_n,i)-D(\mathbf{X}_n^m,i)}\stackrel{\eqref{degree_W_expect}}{=}
\expect{ X(1, 2)-\min\{ X(1,2),m\}} \leq \varepsilon. 
\end{equation}
 
It follows from Fatou's lemma that $\expect{ D(\mathbf{X},i)-D(\mathbf{X}^m,i)} \leq \varepsilon$ also holds.

In order to prove \eqref{matrix_and_degrees_converge} we only need to check
\begin{equation*}
\lim_{n \to \infty} \expect{ f  \left( \mathbf{X}_n^{[k]},
 \left( D(\mathbf{X}_n,i)\right)_{i=1}^k \right)}=
\expect{ f\left( \mathbf{X}^{[k]},
 \left( D(\mathbf{X},i)\right)_{i=1}^k \right)}
\end{equation*}
for any bounded and continuous $f: \A_k \times [0,+\infty)^k \to \R$. This can be easily proved
using \eqref{truncated_toinp}, \eqref{unif_int_epsilon_trunc} and the 
$\varepsilon/3$-argument (see Chapter 1.5 of \cite{reedsimon}). This finishes the proof of 
\eqref{lemma_uniform_integrabiliy_a}.

$ $

Proof of \eqref{lemma_uniform_integrabiliy_b}:
For each $n \in \N$ let $\left( \eta^n_i \right)_{i=1}^{\infty}$ be i.i.d. and uniformly distributed on $[n]$. 
Define the infinite array $\mathbf{X}_n$ by $X_n(i,j):=\hat{X}_n(\eta^n_i,\eta^n_j)$. Now $\mathbf{X}_n$ is vertex exchangeable
and using the vertex exchangeability of $\tilde{\mathbf{X}}_n$ we get
\[ \expect{ X_n(1,2) ; m}=(1-\frac{1}{n}) \expect{ \tilde{X}_n(1,2) ; m} + \frac{1}{n}\expect{ \tilde{X}_n(1,1) ; m}, \]
 end if we combine this with the assumptions of  \eqref{lemma_uniform_integrabiliy_b} we get that
 $\left( X_n(1,2)\right)_{n=1}^{\infty}$  is uniformly integrable.

Note that by \eqref{def_D_X_lim} and the law of large numbers have $D(\mathbf{X}_n,i)=\frac{1}{n} d(\tilde{\mathbf{X}}_n, \eta^n_i)$. Using 
the vertex exchangeability of $\tilde{\mathbf{X}}_n$ we get that the following two $\left( \A_k,\R_+^k \right)$-valued random variables have
the same distribution:
\begin{itemize}
\item  $\left( \mathbf{X}_n^{[k]}, \left( D(\mathbf{X}_n,i)\right)_{i=1}^k \right)$ 
 under the condition  $ \abs{ \{ \eta_1^n, \dots, \eta_k^n \}}=k $ 
\item $\left( \tilde{\mathbf{X}}_n^{[k]}, \left( \frac{1}{n} d(\tilde{\mathbf{X}}_n,i)\right)_{i=1}^k \right)$
\end{itemize}
Let us call this fact $(*)$.

$\mathbf{X}_n \toind \mathbf{X}$ easily follows from $\tilde{\mathbf{X}}_n \toind \mathbf{X}$, $(*)$ and
\begin{equation}\label{separation_eq}
 \lim_{n \to \infty} \prob{ \abs{ \{ \eta_1^n, \dots, \eta_k^n \}}=k} =1,
\end{equation}
 so we can apply
\eqref{lemma_uniform_integrabiliy_a} to obtain \eqref{matrix_and_degrees_converge}. Now using
$(*)$ and \eqref{separation_eq} again we obtain 
\eqref{matrix_and_degrees_converge_finite}.

\end{proof}

\section{Random urn configurations and edge stationarity}
\label{section_stationary_state}

In this section we define a way of constructing random adjacency matrices using random urn configurations 
(the basic idea comes from Section 3.4 of \cite{randomlygrown}). This construction relates edge stationary random
adjacency matrices to ball exchangeable urn models and gives an easy proof of  Lemma \ref{lemma_stationary_distribution}
 using the fact that the distribution
of the $\text{PAG}_{\kappa}(n,m)$ and that of the stationary state of the edge reconnecting model both arise 
from the P\'olya urn model via our construction.

In Subsection \ref{subsection_limits_edge_stac_thm12} we prove
  Theorem \ref{thm_graph_lim_of_edge_stationary} and Theorem \ref{thm_stac_graphlim} using this machinery.

$ $

Let $n,m \in \N$. A random urn configuration with $2m$ balls of $n$ different colors is a probability distribution on $[n]^{[2m]}$, that is
a random function $\Psi: [2m] \to [n]$. If $l \in [2m]$ we say that the $l$'th ball has color $\Psi(l)$.
Let $d(\Psi,i):=\sum_{l=1}^{2m} \ind[ \Psi(l)=i]$ for $i \in [n]$ denote the multiplicity of color $i$ in $\Psi$.

We say that a random urn configuration $\Psi$ is \emph{ball exchangeable} if for all permutations $\tau: [2m] \to [2m]$ we have
\[ \left( \Psi(l) \right)_{l=1}^{2m} \sim \left( \Psi(\tau(l)) \right)_{l=1}^{2m}. \]
$\Psi$ is ball exchangeable if and only if the following property holds: conditioned on the value of the \emph{type vector}
$\left( d(\Psi,i)  \right)_{i=1}^n$, the distribution of $\Psi$ is uniform on the elements of $[n]^{[2m]}$ with this particular
 type vector, more precisely if $\psi \in [n]^{[2m]}$ then
 \begin{equation*}
 \prob{ \Psi=\psi}= \frac{ \prob{ \left(d(\Psi,i)\right)_{i=1}^n =  \left(d(\psi,i)\right)_{i=1}^n}}
{ \left( \frac{ (2m)!}{ \prod_{i=1}^n d(\psi,i)!} \right)}
 \end{equation*}

We say that $\Psi$ is \emph{color exchangeable} if for all permutations $\tau: [n] \to [n]$ we have
\[ \left( \Psi(l) \right)_{l=1}^{2m} \sim \left( \tau(\Psi(l)) \right)_{l=1}^{2m}. \]

To a random urn configuration $\Psi$ we assign a random element $\mathbf{X}$ of $\A_n^m$ by  defining
\begin{equation}\label{from_urn_to_adjmatrix}
  X(i,j):= \sum_{e=1}^m \ind [ \Psi(2e-1)=i, \Psi(2e)=j] + \ind [ \Psi(2e-1)=j, \Psi(2e)=i]
\end{equation}
for all $i,j \in [n]$. In plain words: the colors of the balls correspond to the labels of the vertices and if
 for any $1 \leq e \leq m$ we see a ball of color $i$ at position $2e-1$ and a ball of color $j$ at position $2e$  
then we draw an edge between the vertices $i$ and $j$ in the corresponding labeled multigraph
 (and if $i=j$ then we draw a loop edge at vertex $i$).

With the definition \eqref{from_urn_to_adjmatrix} we have
$\prob{ d(\mathbf{X},i)= d(\Psi,i)}=1$.
 It is easy to see that all probability measures on $\A_n^m$ arise this way.

  If $\Psi$ is color exchangeable then
 $\mathbf{X}$ is vertex exchangeable. All vertex exchangeable probability measures on $\A_n^m$ arise this way.

If $\Psi$ is ball exchangeable then for all $B \in \A_n^m$ we have
\begin{equation}\label{edge_stationarity_eq}
\prob{ \mathbf{X}=B}=
\frac{
\prob{ \left(d(\mathbf{X},i)\right)_{i=1}^n =  \left(d(B,i)\right)_{i=1}^n}}
{ \left( \frac{ (2m)!}{ \prod_{i=1}^n d(B,i)!} \right)}
 \frac{m! 2^{m'(B)}}{\left(\prod_{i<j} B(i,j)!\right)
  \left(\prod_{i=1}^n \frac{B(i,i)}{2}! \right)}
\end{equation}
where $m'(B)$ denotes the number of non-loop edges. The first term in \eqref{edge_stationarity_eq} is $\prob{ \Psi=\psi}$ for
some $\psi$ that produces $B$ via \eqref{from_urn_to_adjmatrix}, the second term is the number of elements of $[n]^{[2m]}$ that produce
$B$ via \eqref{from_urn_to_adjmatrix}.

Recalling the definition of the configuration model (see Section \ref{introduction}) we can see that if we generate  $\mathbf{X}$
using \eqref{from_urn_to_adjmatrix} from a ball exchangeable urn configuration $\Psi$ with a given  degree sequence
$\left(d_i\right)_{i=1}^n$ then we in fact uniformly choose one from the the set of possible matchings of the stubs where vertex $i \in [n]$
has $d_i$ stubs. Thus \eqref{edge_stationarity_eq} holds for a random element $\mathbf{X}$ of $\A_n^m$ if and only if 
the distribution of $\mathbf{X}$ is edge stationary. It is easy to see that all edge-stationary probability distributions 
on $\A_n^m$ arise from ball exchangeable
distributions on $[n]^{[2m]}$ via \eqref{from_urn_to_adjmatrix}.

Now we define two different dynamics on random urn configurations:
\begin{itemize}
\item \emph{The P\'olya urn model:} Fix $\kappa \in (0,+\infty)$. Let $\Psi_L$ be a random element of $[n]^{[L]}$. Given
$\Psi_L$ we generate a random element of $[n]^{[L+1]}$ which we denote by $\Psi_{L+1}$ in the following way: let
$\Psi_{L+1}(l) := \Psi_L(l)$ for all $l \in [L]$ and
\[ \forall \, i \in [n]: \;\;  \condprob{\Psi_{L+1}(L+1)=i}{ \Psi_L}= \frac{ d( \Psi_L, i)+\kappa}{ L+ n\kappa} \]
\item \emph{The ball replacement model:} Fix $\kappa \in (0,+\infty)$. Let $\Psi_T$ be a random element of
$[n]^{[2m]}$. Given $\Psi_T$ we generate a random element of $[n]^{[2m]}$
  which we denote by $\Psi_{T+1}$ in the following way: let $\xi_T$ denote a uniformly chosen element of $[2m]$.
  For all $l \in [2m] \setminus \xi_T$ let $\Psi_{T+1}(l):=\Psi_{T}(l)$ and
\begin{equation}\label{ball_replacement_rule}
 \forall \, i \in [n]:\;\;  \condprob{\Psi_{T+1}(\xi_T)=i}{ \Psi_T, \xi_T}= \frac{ d( \Psi_T, i)+\kappa}{ 2m+ n\kappa}
 \end{equation}
\end{itemize}
It is well-known that if we start with an empty urn $\Psi_0$ and repeatedly apply the P\'olya urn scheme to get
$\Psi_L$ for $L=1,2,\dots,2m$ then the distribution of $\Psi_{2m}$ is of the following form:
\begin{equation}\label{polya_distribution}
 \forall\, \psi \in [n]^{[2m]}:\;\; \prob{ \Psi_{2m}=\psi}=\frac{\prod_{i=1}^n
\prod_{j=1}^{d(\psi,i)}(\kappa+j-1)}{ \prod_{j=1}^{2m}(\kappa n
+j-1)}
\end{equation}
Thus the distribution of $\Psi_{2m}$ is ball and color exchangeable.
The $\text{PAG}_{\kappa}(n,m)$ (defined in Section \ref{introduction}) is in fact the random multigraph
 obtained as  the image of the random urn configuration \eqref{polya_distribution}
 under the mapping \eqref{from_urn_to_adjmatrix}.

 The ball replacement model is  an
$[n]^{[2m]}$-valued Markov chain, which is irreducible and aperiodic with
unique stationary distribution \eqref{polya_distribution}: if we delete the $\xi_T$'th ball from $\Psi_{2m}$, then by
 ball exchangeability
  the distribution of the resulting $[n]^{[2m-1]}$-valued random variable is the same as deleting the $2m$'th ball:
   P\'olya-$\Psi_{2m-1}$. Thus replacing the removed $\xi_T$'th ball with a new one according
   to \eqref{ball_replacement_rule} we get a $[n]^{[2m]}$-valued random variable
   with  P\'olya-$\Psi_{2m}$ distribution again by ball exchangeability. 

Now consider the ball replacement Markov chain $\Psi_T$, $T=0,1,\dots$ with $\Psi_0$ being an arbitrary $[n]^{[2m]}$-valued random variable.
If we use the mapping \eqref{from_urn_to_adjmatrix} to create $\mathbf{X}(T)$ from $\Psi_T$, then it is easily seen that the resulting $\A_n^m$-valued
stochastic process $\mathbf{X}(T)$, $T=0,1,\dots$ evolves according to the rules of the edge reconnecting Markov chain defined in Section
\ref{introduction}. Some consequences of this fact:
\begin{itemize}
\item
If the distribution of $\Psi_0$ is ball exchangeable then $\Psi_T$ is also ball exchangeable for all $T$,
thus if
$\mathbf{X}(0)$ is edge stationary then $\mathbf{X}(T)$ is also edge stationary for all $T$
(hence the name ``edge stationarity").
\item The distribution \eqref{polya_distribution} is stationary for the ball replacement model, thus
the image of this distribution under the mapping \eqref{from_urn_to_adjmatrix} is the unique stationary
distribution of the edge reconnecting model. Lemma \ref{lemma_stationary_distribution} follows from
 \eqref{polya_distribution} and \eqref{edge_stationarity_eq}.
\end{itemize}

\subsection{Limits of edge stationary multigraph sequences}
\label{subsection_limits_edge_stac_thm12}

The key result of this subsection is Lemma \ref{lemma_edge_stac_graphlim} which can be roughly summarized as follows:
in a large dense edge stationary random multigraph the number of edges connecting the vertices $v$ and $w$ has Poisson distribution
 with parameter proportional to $d(v)d(w)$.
Given  Lemma \ref{lemma_edge_stac_graphlim}  the proof of Theorem \ref{thm_graph_lim_of_edge_stationary} is straightforward 
and the proof of
 Theorem \ref{thm_stac_graphlim} reduces to a 
limit theorem which states that the rescaled number of balls with color $1,2,\dots,k$ in the P\'olya urn model
converge in distribution to i.i.d. random variables with Gamma distribution.

$ $

\begin{lemma}\label{lemma_edge_stac_graphlim}
 Let $F: [0,+\infty) \to [0,1]$ denote the cumulative distribution
function of a nonnegative random variable $Z$. 
Let $F^{-1}(u):=\min \{x:F(x)\geq u \}$.
Let $Z_1,Z_2, \dots$ be i.i.d. random variables with $Z_i \sim Z \sim F^{-1}(U_i)$ (where $U_i$ are uniform on $[0,1]$).

If $\mathbf{X}_n$ is an $\A_n$-valued random variable for $n=1,2,\dots$, moreover the distribution of
$\mathbf{X}_n$ is vertex exchangeable and edge stationary, and 
\begin{equation}\label{edge_stac_rho_asymp}
 \frac{2m(\mathbf{X}_n)}{ n^2} \toinp \rho, \qquad   n \to \infty,
\end{equation}
 where $0< \rho<+\infty$ is positive real parameter, moreover  for all $k \in \N$ we have
 \begin{equation}\label{edge_stac_conv_Z_assumpt}
 \left( \frac{1}{n} d(\mathbf{X}_n,i) \right)_{i=1}^k \toind \left( Z_i \right)_{i=1}^k, \qquad   n \to \infty
 \end{equation}
 then $\mathbf{X}_n \toinp W$ where
 \begin{equation}
  \label{edge_stac_graphlim}
W(x,y,k) =
 \left\{
\begin{array}{ll}
\mypoi(k,\frac{F^{-1}(x)F^{-1}(y)}{\rho}) & \mbox{ if $x \neq y$}\\
\ind [2|k]\cdot \mypoi \left(\frac{k}{2},\frac{F^{-1}(x)F^{-1}(y)}{2\rho} \right) & \mbox{ if $x= y$}
\end{array} \right.
 \end{equation}
 \end{lemma}

\begin{proof}
The infinite random array $\mathbf{X}_W$ (see Definition \ref{def_X_W}) can be alternatively defined in the
following way:
Let $\left(X_W(i,j)\right)_{i\leq j}$ be conditionally independent given
 $\left(Z_i\right)_{i \in \N}$ with conditional distribution
 $X_W(i,j) \sim \text{POI} \left( \frac{Z_i Z_j}{\rho} \right)$  if $i<j$ and
 $\frac{X_W(i,i)}{2} \sim \text{POI}\left( \frac{Z_i Z_i}{2\rho} \right)$.

 If $A \in \A_k$ let $A^*$ denote the following modified matrix: $A^*(i,j):=A(i,j)$ if $i \neq j$ but $A^*(i,i):=\frac{A(i,i)}{2}$.
Thus $A^*(i,i)$ is the number of loop edges at vertex $i$.  

Let
$m_{[k]}:=\frac{1}{2} \sum_{i,j} A(i,j)$. Define
\begin{equation*}
\mypoi(A,\left(z_i\right)_{i=1}^k,\rho ):=
\exp \left( \frac{-1}{2\rho} \left(\sum_{i=1}^k z_i \right)^2 \right)
\cdot \prod_{i \leq j} \frac{1}{A^*(i,j)!} \cdot \prod_{i=1}^k \left( z_i \right)^{d(A,i)} \cdot \rho^{-m_{[k]}} \cdot
2^{-\sum_{i=1}^k A^*(i,i)}
\end{equation*}
 By   \eqref{X_W_indep_prod_formula}  and \eqref{edge_stac_graphlim} we have
\begin{equation}\label{poi_graphlim_W_formula_okt25}
\condprob{ \mathbf{X}_W^{[k]} =A}{ \left(Z_i\right)_{i=1}^k}=
\prod_{i=1}^k \prod_{j=i}^k
\mypoi \left( A^*(i,j), \frac{ Z_{i} \cdot Z_{j} }{ \rho \cdot (1+\ind[i=j]) } \right)=
\mypoi(A,\left(Z_i\right)_{i=1}^k,\rho)
.
\end{equation}
By Lemma \ref{lemma_homkonv_konv_in_dist}  we only need to show that  we have
\begin{equation}\label{edge_stac_conv_prob_A}
\forall\, k\in \N, \; \forall \, A \in \A_k:\;
 \lim_{n \to \infty} \prob{\mathbf{X}_n^{[k]}=A} =
 \prob{\mathbf{X}_W^{[k]}=A}
\end{equation}
  in order to prove $\mathbf{X}_n \toinp W$.

Let $\left( d_i \right)_{i=1}^n$ denote an arbitrary 
degree sequence with $m=\frac{1}{2} \sum_{i=1}^n d_i$ and denote by
\begin{equation}\label{def_eq_z_i_rho_n}
   z_i:=\frac{d_i}{n}, \quad \rho_n:= \frac{2m}{n^2}.
\end{equation}

Fix $\varepsilon>0$ and $A \in \A_k$. We are going to prove that if 
\begin{equation}\label{assumptions_eps}
\varepsilon \leq \rho_n \leq \varepsilon^{-1}, \qquad \forall\;  i \in [k]\, :\;z_i \leq \varepsilon^{-1}
\end{equation}
 then 
\begin{gather}
\label{edge_stac_condprob_given_degrees}
\condprob{ \mathbf{X}_n^{[k]} =A}{ \left( d(\mathbf{X}_n,i) \right)_{i=1}^k=\left( d_i \right)_{i=1}^k, \;
\frac{2m(\mathbf{X}_n)}{ n^2}=\rho_n }=\\
\label{edge_stac_poi_approx_formula}
\mypoi(A,\left(z_i\right)_{i=1}^k,\rho_n)
+\text{Err}(n,A,\varepsilon)
\end{gather}
with  $\lim_{n \to \infty} \text{Err}(n,A, \varepsilon)  =0$.
 We adopt the convention that the value of $\text{Err}(n,A, \varepsilon)$
might change from line to line.

First we assume that $\eqref{edge_stac_condprob_given_degrees}=\eqref{edge_stac_poi_approx_formula}$ holds
 under the condition
\eqref{assumptions_eps},  and deduce \eqref{edge_stac_conv_prob_A} from it.
 Define the events $B_n^{\varepsilon}$ and $B^{\varepsilon}$ by
\begin{align*}
B_n^{\varepsilon}:=& \{ \varepsilon \leq \frac{2m(\mathbf{X}_n)}{ n^2} \leq \varepsilon^{-1}, \; \;  
\forall\, i \in [k]:\;  \frac{1}{n} d(\mathbf{X}_n,i) \leq \varepsilon^{-1} \} \\
B^{\varepsilon}:=& \{ \varepsilon \leq \rho \leq \varepsilon^{-1}, \; \;  
\forall\, i \in [k]:\;  Z_i \leq \varepsilon^{-1} \} 
\end{align*}
Using the Portmanteau theorem and \eqref{edge_stac_rho_asymp}, \eqref{edge_stac_conv_Z_assumpt}
 we get that $\limsup_{n \to \infty} \prob{B_n^{\varepsilon}} \leq \prob{B^{\varepsilon}}$.

\begin{gather}
\abs{\prob{\mathbf{X}_n^{[k]}=A} -
 \prob{\mathbf{X}_W^{[k]}=A}} \stackrel{ \eqref{poi_graphlim_W_formula_okt25}}{=} 
\abs{\prob{\mathbf{X}_n^{[k]}=A} -
\expect{ \mypoi\left(A, \left( Z_i\right)_{i=1}^k, \rho \right)}  } 
 \leq \\
 \label{edge_stac_bors_egy}
\abs{\expect{ \mypoi\left(A, \left( \frac{1}{n} d(\mathbf{X}_n,i)\right)_{i=1}^k, \frac{2m(\mathbf{X}_n)}{ n^2} \right);\, B_n  }-
\expect{ \mypoi\left(A, \left( Z_i\right)_{i=1}^k, \rho \right)}} +\\
\label{edge_stac_bors_ketto}
 \text{Err}(\varepsilon,A,n) + (1-\prob{ B_n^{\varepsilon}})
\end{gather}
By  \eqref{edge_stac_rho_asymp}, \eqref{edge_stac_conv_Z_assumpt}, $\lim_{n \to \infty} \text{Err}(n,A, \varepsilon)  =0$ 
and the fact that $\mypoi\left(A, \left( \cdot \right)_{i=1}^k, \cdot \right)$ is a bounded continuous function on the domain
\eqref{assumptions_eps}
 we obtain
\[\limsup_{n \to \infty}  \eqref{edge_stac_bors_egy} \leq 1-\prob{ B^{\varepsilon}}  \quad \text{ and } \quad
\limsup_{n \to \infty} \eqref{edge_stac_bors_ketto}\leq 1-\prob{ B^{\varepsilon}}.\]
Now $\prob{B^{\varepsilon}} \to 1$ as $\varepsilon \to 0$, from which \eqref{edge_stac_conv_prob_A}
and the statement of the lemma follows under the assumption that \eqref{assumptions_eps} implies
$\eqref{edge_stac_condprob_given_degrees}=\eqref{edge_stac_poi_approx_formula}$.
\end{proof}

\medskip

\begin{proof}[Proof of $\eqref{assumptions_eps} \implies \eqref{edge_stac_condprob_given_degrees}=\eqref{edge_stac_poi_approx_formula}$:]

$ $

We are using random urn configurations to generate $\mathbf{X}_n$.
 Let $\Psi_n$ denote the ball and color exchangeable $[n]^{[2m]}$-valued random variable
with $\left(d(\Psi,i)\right)_{i=1}^n =  \left(d_i\right)_{i=1}^n$, thus  $\Psi_n$ is uniformly distributed
on the set of urn configurations with this type vector.
$\mathbf{X}_n$ can be generated from $\Psi_n$ via \eqref{from_urn_to_adjmatrix}. To determine the distribution of
$\mathbf{X}_n^{[k]}$ we only need to know the positions of the balls of color $i \in [k]$.
We paint the rest of the balls ''grey``.  Let
\begin{equation*}
m_{[k]}:=\frac{1}{2} \sum_{i,j} A(i,j),\quad  d_{[k]}:=\sum_{i=1}^k d_i, \quad  
m_g:= m - d_{[k]} +m_{[k]}.
\end{equation*}
Thus $m_g$
denotes the number of edges of the multigraph spanned by grey vertices.

 In order to prove $\eqref{edge_stac_condprob_given_degrees}=\eqref{edge_stac_poi_approx_formula}$  we first
 give an explicit formula for \eqref{edge_stac_condprob_given_degrees}.

 The number of grey balls is $2m-d_{[k]}$.
The number of all urn configurations with type vector $(d_1,\dots,d_k,2m-d_{[k]})$ is
\begin{equation}\label{tutu_osszes}
 \frac{ (2m)!}{ \left( \prod_{i=1}^k d_i! \right) \cdot
 (2m- d_{[k]})!}
 \end{equation}

The number of urn configurations with type vector $(d_1,\dots,d_k,2m-d_{[k]})$
 for which $\mathbf{X}_n^{[k]} =A$ is
\begin{equation}\label{tutu_kedvezo}
\frac{m! \cdot 2^{m - m_g - \sum_{i=1}^k A^*(i,i)} }
{\left( \prod_{i \leq j} A^*(i,j)!
\right)\cdot \left( \prod_{i=1}^k (d_i -d(A,i))! \right)\cdot m_g!}
\end{equation}
Thus
 $\eqref{edge_stac_condprob_given_degrees}= \frac{\eqref{tutu_kedvezo}}{\eqref{tutu_osszes}}$.
Our aim is to prove $\frac{\eqref{tutu_kedvezo}}{\eqref{tutu_osszes}}=\eqref{edge_stac_poi_approx_formula}$:
after dividing both sides of this equality by $\prod_{i \leq j} \frac{1}{A^*(i,j)!} \cdot 2^{-\sum_i A^*(i,i)}$ we only need to prove
\begin{gather}\label{edge_stac_we_need_lhs}
\frac{m! \cdot \left( \prod_{i=1}^k d_i! \right) \cdot \left( 2m-
d_{[k]} \right)! \cdot 2^{m -m_g}}{ \prod_{i=1}^k (d_i -d(A,i))! \cdot m_g! \cdot (2m)!}
= \\
\label{edge_stac_we_need_rhs}
\exp \left( \frac{-1}{2\rho_n} \left(\sum_{i=1}^k z_i \right)^2 \right)
 \cdot \prod_{i=1}^k \left( z_i \right)^{d(A,i)} \cdot \rho_n^{-m_{[k]}}
+\text{Err}(n, A,\varepsilon)
\end{gather}

Now we rewrite \eqref{edge_stac_we_need_lhs}:
\begin{multline}\label{last_exact_formula}
\eqref{edge_stac_we_need_lhs}=
\left(\prod_{i=1}^k \prod_{l=1}^{d(A,i)} (d_i-d(A,i)+l)
\right)\cdot \left( \prod_{l=1}^{m-m_g} (m_g+l) \right) \cdot
\frac{2^{m-m_g}}{\prod_{l=1}^{d_{[k]}}(2m-d_{[k]}+l)}=\\
\left(\prod_{l=1}^{m-m_g} \frac{ 2m_g +2l}{2m-d_{[k]}+l} \right)
\cdot \frac{
\prod_{i=1}^k \prod_{l=1}^{d(A,i)} (d_i-d(A,i)+l)
}{\prod_{l=1}^{m_{[k]}} (2m-m_{[k]}+l)}
\end{multline}

Now we  approximate various terms that appear in the right hand side of \eqref{last_exact_formula} using
  our assumpions \eqref{assumptions_eps}:
\begin{align}
\label{bound_z_1} \prod_{l=1}^{m-m_g} \frac{ 2m_g +2l}{2m-d_{[k]}+l} =\left(\prod_{l=1}^{d_{[k]}} \frac{2m-2l}{2m-l}\right)&
\cdot
\left( 1+ \frac{1}{n} \text{Err}(A,\varepsilon)\right) \\
\label{bound_z_2}
\prod_{l=1}^{m_{[k]}} (2m-m_{[k]}+l)=(2m)^{m_{[k]}}&
\cdot
\left( 1+ \frac{1}{n^2} \text{Err}(A,\varepsilon)\right)
\end{align}
where $0\leq \abs{\text{Err}(A,\varepsilon)}<+\infty$ is independent of $n$.

Let $d^*= \min \{ d_i \, : \; i \in [k], \; d(A,i)>0 \}$. We consider two cases separately: 

If $d^* \leq n^{1/2}$ then using \eqref{assumptions_eps}
 it is easy to see that $\eqref{last_exact_formula} \leq \text{Err}(A,\varepsilon) n^{-1/2}$ and also
$\mypoi(A,\left(z_i\right)_{i=1}^k, \rho_n ) \leq \text{Err}(A,\varepsilon) n^{-1/2}$, so  
$\eqref{edge_stac_we_need_lhs}=\eqref{edge_stac_we_need_rhs}$ holds when $d^* \leq n^{1/2}$.

If $d^* > n^{1/2}$ then we have
\begin{equation}\label{bound_z_3}
 \prod_{i=1}^k \prod_{l=1}^{d(A,i)} (d_i-d(A,i)+l)= \left( \prod_{i=1}^k (d_i)^{d(A,i)} \right) 
\left(1+ \frac{1}{\sqrt{n}}\text{Err}(A,\varepsilon)\right).
\end{equation}
Putting \eqref{bound_z_1}, \eqref{bound_z_2} and \eqref{bound_z_3} together we get

\begin{multline*}
\eqref{last_exact_formula}= \left(\prod_{l=1}^{d_{[k]}} \frac{2m-2l}{2m-l}\right)\cdot
\frac{\prod_{i=1}^k (d_i)^{d(A,i)}}{(2m)^{m_{[k]}}} \cdot
 \left( 1+ \text{Err}(n,A,\varepsilon)\right) \stackrel{\eqref{def_eq_z_i_rho_n}}{=}\\
\left(\prod_{l=1}^{d_{[k]}} \frac{1-\frac{2l}{ n^2 \rho_n}}{1-\frac{l}{ n^2 \rho_n }}\right)\cdot
 \frac{\prod_{i=1}^k \left( n\cdot z_i \right)^{d(A,i)}}{ (n^2  \rho_n)^{m_{[k]}}}\cdot
 \left(1+ \text{Err}(n,A,\varepsilon)\right)
\stackrel{\eqref{assumptions_eps}}{=} \\
\exp \left( \frac{-1}{2\rho_n} \left(\sum_{i=1}^k z_i \right)^2 \right)
 \cdot \prod_{i=1}^k \left( z_i \right)^{d(A,i)} \cdot \rho_n^{-m_{[k]}}
+\text{Err}(n, A,\varepsilon)
\end{multline*}
This completes the proof of $\eqref{edge_stac_we_need_lhs}=\eqref{edge_stac_we_need_rhs}$.
\end{proof}

$ $

\begin{proof}[Proof of Theorem \ref{thm_graph_lim_of_edge_stationary}]
Given $\mathbf{X}_n$ for every $n \in \N$ let us define the vertex exchangeable random adjacency matrix 
$\tilde{\mathbf{X}}_n$ in the following way:
let $\pi_n$ denote a uniformly distributed permutation $\pi_n: [n] \hookrightarrow [n]$, independent from  $\mathbf{X}_n$.  
Let 
\begin{equation}\label{def_eq_permut_exch_X_tilde}
\left(\tilde{X}_n(i,j)\right)_{i,j=1}^n:= \left( X_n(\pi_n(i),\pi_n(j) \right)_{i,j=1}^n.
\end{equation}

Then $\tilde{\mathbf{X}}_n$ is indeed vertex exchangeable, moreover $\prob{t_=(F,\tilde{\mathbf{X}}_n)=t_=(F,\mathbf{X}_n)}=1$ for every 
$F \in \M$, so $\mathbf{X}_n \toinp W$ is  equivalent to $\tilde{\mathbf{X}}_n \toinp W$, 
which is in turn equivalent to $\tilde{\mathbf{X}}_n \toind \mathbf{X}_W$
by Lemma \ref{lemma_homkonv_konv_in_dist}. By \eqref{def_eq_permut_exch_X_tilde} the conditions \eqref{egyenletes_int_zuri_1} and
\eqref{egyenletes_int_zuri_2} are equivalent to the uniform integrability of $\left(\tilde{X}_n(1,2)\right)_{n=1}^{\infty}$
and $\left(\tilde{X}_n(1,1)\right)_{n=1}^{\infty}$, respectively, thus we can apply Lemma 
\ref{lemma_uniform_integrabiliy_whole}/\eqref{lemma_uniform_integrabiliy_b} to deduce that for all $k \in \N$
\begin{equation}\label{degrees_converge_to_indep_DW}
 \left( \frac{1}{n} d(\tilde{\mathbf{X}}_{n},i)\right)_{i=1}^k 
\toind
 \left( D(\mathbf{X}_W,i)\right)_{i=1}^k.
\end{equation}

Note that by \eqref{def_eq_F_W}, Definition \ref{def_X_W} and \eqref{degree_W_loln} we have that
 $\left( D(\mathbf{X}_W,i)\right)_{i=1}^k$ are i.i.d. with 
probability distribution function
$ F_W(z)$.

Now we are going to prove that  $\frac{2m(\tilde{\mathbf{X}}_n)}{n^2} \toinp \rho(W)$.
In order to do so we define the truncated adjacency matrix $\tilde{\mathbf{X}}_n^l$ by
$\tilde{X}_n^l(i,j):=\min \{\tilde{X}_n(i,j), l\} $ and the truncated multigraphon
$W^l$ which satisfies $\left(X_{W^l}(i,j)\right)_{i,j=1}^{\infty} \sim \left( \min\{ X_{W}(i,j),l\}  \right)_{i,j=1}^{\infty}$.

Now we show that if we fix $l \in \N$ then
 $\frac{2m(\tilde{\mathbf{X}}^l_n)}{n^2} \toinp \rho(W^l)$. 

The equations marked by $(*)$ below are true by exchangeability:
\begin{multline*}
 \lim_{n \to \infty} \expect{ \frac{1}{n} \sum_{i=1}^n \frac{1}{n}d(\tilde{\mathbf{X}}^l_n,i)} \stackrel{(*)}{=}
 \lim_{n \to \infty} \expect{\frac{1}{n} d(\tilde{\mathbf{X}}_n^l,1)} \stackrel{\eqref{degrees_converge_to_indep_DW}}{=} 
\expect{ D(\mathbf{X}_{W^l},1)} \stackrel{\eqref{degree_W_expect}}{=}\rho(W^l). \\
 \lim_{n \to \infty} \var{ \frac{1}{n} \sum_{i=1}^n \frac{1}{n}d(\tilde{\mathbf{X}}^l_n,i)}=
\lim_{n \to \infty} \frac{1}{n^2} \sum_{i,j=1}^n \cov{ \frac{1}{n}d(\tilde{\mathbf{X}}^l_n,i)} {\frac{1}{n}d(\tilde{\mathbf{X}}^l_n,j)}
 \stackrel{(*)}{=} \\
 \lim_{n \to \infty} 
\left( \frac{1}{n} \var{ \frac{1}{n}d(\tilde{\mathbf{X}}^l_n,1)}+ \frac{n-1}{n} 
\cov{ \frac{1}{n}d(\tilde{\mathbf{X}}^l_n,1) }{ \frac{1}{n}d(\tilde{\mathbf{X}}^l_n,2)} \right) 
\stackrel{\eqref{degrees_converge_to_indep_DW}}{=} 0.
\end{multline*}
Having established $\forall \, l: \; \frac{2m(\tilde{\mathbf{X}}^l_n)}{n^2} \toinp \rho(W^l)$, the relation
 $\frac{2m(\tilde{\mathbf{X}}_n)}{n^2} \toinp \rho(W)$ follows from 
\[ \lim_{l \to \infty} \rho(W^l) = \rho(W), \qquad 
 \forall \, \varepsilon>0 : \;
\lim_{l \to \infty} \sup_{n \in \N} 
\prob{ \frac{2m(\tilde{\mathbf{X}}_n)}{n^2}-\frac{2m(\tilde{\mathbf{X}}^l_n)}{n^2} \geq \varepsilon} 
\stackrel{\eqref{egyenletes_int_zuri_1},
\eqref{egyenletes_int_zuri_2} }{=} 0, \]
and the $\varepsilon/3$-argument.
So conditions \eqref{edge_stac_rho_asymp} and \eqref{edge_stac_conv_Z_assumpt}  are satisfied, thus
we can apply  Lemma \ref{lemma_edge_stac_graphlim} to show that $\tilde{\mathbf{X}}_n \toinp \hat{W}$, where
$\hat{W}$ is of form \eqref{edge_stac_lim_graphon}.
\end{proof}

$ $

\begin{proof}[Proof of Theorem \ref{thm_stac_graphlim}]
The distribution \eqref{stationary} arises from the P\'olya-$\Psi^n_{2m}$ urn model \eqref{polya_distribution}
 with $2m$ balls and $n$ colors  via
\eqref{from_urn_to_adjmatrix}. The distribution \eqref{polya_distribution} is ball and color exchangeable, so
$\mathbf{X}_n$ is vertex exchangeable and edge stationary. If we want to prove  Theorem \ref{thm_stac_graphlim} then
 by Lemma \ref{lemma_edge_stac_graphlim} we only need to show
that \eqref{edge_stac_conv_Z_assumpt} holds for all $k \in \N$
where $\left( Z_i \right)_{i \in \N}$ are i.i.d. with density function $\mygamma(x,\kappa,\frac{\kappa}{\rho})$ (see \eqref{def_mygamma}).
 We may use the method of moments to prove convergence in distribution, since
the Gamma distribution is uniquely determined by its moments. Thus we need to show that if
$\nu_1,\dots, \nu_k \in \N$ then
\begin{equation*}
\lim_{n \to \infty} \expect{ \prod_{i=1}^k \left( \frac{1}{n} d( \Psi^n_{2m(n)}, i) \right)^{\nu_i}}=
\expect{ \prod_{i=1}^k Z_i^{\nu_i}}=
\prod_{i=1}^k \left( \frac{\rho}{\kappa} \right)^{\nu_i} \cdot \prod_{j=1}^{\nu_i} \left(\kappa+ j-1\right).
\end{equation*}
Fix $k$ and $\nu_i,\, i \in [k]$.
Let $\nu=\sum_{i=1}^k \nu_i$ and denote by $\psi$ a particular element of $[k]^{[\nu]}$ with type vector $(\nu_1,\dots,\nu_k)$.
By the construction of the P\'olya-$\Psi^n_{2m}$ distribution we have
\begin{equation*}
\prob{ \forall\, l \in [\nu]:\;  \Psi^n_{2m}(l)=\psi(l)}=\frac{\prod_{i=1}^k
\prod_{j=1}^{\nu_i}(\kappa+j-1)}{ \prod_{j=1}^{\nu}(\kappa n +j-1)}=\Ordo(n^{-\nu})
\end{equation*}
Denote by $\ul{\nu}:= \{ (i,j)\,:\, i \in [k], \, j \in [\nu_i] \}$. The number of functions $f: \ul{\nu} \to [2m]$ with
$\abs{\mathcal{R}(f)}=N$ is $\Ordo((2m(n))^N)=\Ordo( n^{2N}) $ if $ 1 \leq N \leq \nu$.
\begin{multline*}
\lim_{n \to \infty} \expect{ \prod_{i=1}^k \left( \frac{1}{n} d( \Psi^n_{2m(n)}, i) \right)^{\nu_i}}=\\
\lim_{n \to \infty} \frac{1}{n^{\nu}}
\sum_{f: \ul{\nu} \to [2m]} \prob{ \forall \, (i,j) \in \ul{\nu}: \;\; \Psi_{2m(n)}^n(f(i,j))=i}=\\
\lim_{n \to \infty} \frac{1}{n^{\nu}}
\sum_{f: \ul{\nu} \hookrightarrow [2m]} \prob{ \forall\, (i,j) \in \ul{\nu}: \;\; \Psi_{2m(n)}^n(f(i,j))=i}
+ \lim_{n \to \infty} \frac{1}{n^{\nu}} \sum_{N=1}^{\nu -1} \Ordo(n^{2N}) \Ordo(n^{-N})\stackrel{(\ast)}{=}\\
\lim_{n \to \infty} \frac{\prod_{k=1}^{\nu} (2m(n)-k+1)  }{n^{\nu}} \prob{ \forall\, l \in [\nu]: \;  \Psi^n_{2m}(l)=\psi(l)}
\stackrel{\eqref{edge_stac_rho_asymp}}{=}
\prod_{i=1}^k \left( \frac{\rho}{\kappa} \right)^{\nu_i} \cdot \prod_{j=1}^{\nu_i} \left(\kappa+ j-1\right).
\end{multline*}
The equation $(\ast)$ holds true by ball exchangeability.
\end{proof}

\end{document}